\documentclass[12pt]{amsart}
\usepackage{amssymb}
\usepackage{amsfonts}
\usepackage{amsmath,amsthm,amsfonts,amscd,amssymb,mathdots,amstext}
\usepackage{fullpage}
\usepackage[all]{xy}
 \usepackage[usenames,dvipsnames]{color}

\usepackage{graphicx}
\usepackage{bbding}
\usepackage[misc]{ifsym}

\theoremstyle{plain}
\newtheorem{thm}{Theorem}[section]
\newtheorem{cor}[thm]{Corollary}
\newtheorem{prop}[thm]{Proposition}
\newtheorem{lem}[thm]{Lemma}

\newcounter{theoremintro}
\newtheorem{theoremi}[theoremintro]{Theorem}

\theoremstyle{definition}
\newtheorem{rem}[thm]{Remark}

\newtheorem{defn}[thm]{Definition}

\newtheorem{ques}[thm]{Question}

\newcommand{\bC}{{\mathbb{C}}}

\newcommand{\bN}{{\mathbb{N}}}

\newcommand{\bR}{{\mathbb{R}}}

  \newcommand{\B}{{\mathcal{B}}}

\renewcommand{\H}{{\mathcal{H}}}

  \newcommand{\K}{{\mathcal{K}}}

\renewcommand{\P}{{\mathcal{P}}}

\renewcommand{\phi}{\varphi}
\newcommand{\upchi}{{\raise.35ex\hbox{$\chi$}}}

\newcommand{\doplus}{\operatornamewithlimits{\oplus}}
\newcommand{\rd}{\operatorname{d}}

\begin{document}

%%%%%%%%%%%%%%%%%%%%%%%%%%%%%%%%%%%%%%%%%%%
\title{On the power set of quasinilpotent operators}

\author[Y. Q. Ji ]{Youqing Ji }
\author[Y. H. Zhang]{Yuanhang Zhang}

\address{Youqing JI, School of Mathematics\\Jilin University\\Changchun 130012\\P.R. CHINA}\email{jiyq@jlu.edu.cn}
\address{Yuanhang Zhang, School of Mathematics\\Jilin University\\Changchun 130012\\P.R. CHINA}\email{zhangyuanhang@jlu.edu.cn}

\thanks{Youqing Ji is supported in part by NNSF of China (Nos.:12271202; 12031002); Yuanhang Zhang is supported in part by NNSF of China (No.:12071174).}
\begin{abstract}
For a quasinilpotent operator $T$ on a separable Hilbert space $\H$, Douglas and Yang define $k_x=\limsup\limits_{\lambda\rightarrow 0}\frac{\ln\|(\lambda-T)^{-1}x\|}{\ln\|(\lambda-T)^{-1}\|}$ for each nonzero vector $x$, and call $\Lambda(T)=\{k_x:x\ne 0\}$ the {\em power set} of $T$.
In this paper, we prove that $\Lambda(T)$ is right closed, that is, $\sup \sigma\in\Lambda(T)$ for each nonempty subset $\sigma$ of $\Lambda(T)$. Moreover, for any right closed subset $\sigma$ of $[0,1]$ containing $1$, we show that there exists a quasinilpotent operator $T$ with $\Lambda(T)=\sigma$. Finally, we prove that the power set of $V$, the Volterra operator on $L^2[0,1]$,
is $(0,1]$.
\end{abstract}
%\date{\today}
\maketitle

%%%%%%%%%%%%%%%%%%%%%%%%%%%%%%%%%%%%%%%%%%%
\section{Introduction}

Let  $\H$ be a separable infinite dimensional complex Hilbert space, and  $\B(\H)$ be the algebra of bounded linear operators on $\H$. For $T\in \B(\H)$, let
 $\sigma(T)$ be  the spectrum of $T$ and $\rho(T)$ be the resolvent set.
 Say $T$ is  {\em quasinilpotent} if $\sigma(T)=\{0\}$.

To detect a finer understanding on the structure of quasinilpotent operators, Douglas and Yang \cite{DY3} introduced a new way to analyze the resolvent set. Let us recall Douglas and Yang's idea briefly. Given $T\in \B(\H)$ and $x\in \H\setminus\{0\}$, since $g_x(\lambda)=\|(\lambda-T)^{-1}x\|^2>0$ for all $\lambda\in \rho(T)$, $g_x(\lambda)|d\lambda|^2$ defines a Hermitian metric on $\rho(A)$.
This metric could reflect the properties of $T$ and its resolvent action on elements of $\H$ (cf. \cite{DY2,DY3,LY19,TY20}).
The notion of the {\em power set} is defined to gauge the ``blow-up" rate of this metric at 0 for quasinilpotent operators.
For convenience, we recall it below.

\begin{defn}
Suppose that $T\in\B(\H)$ is quasinilpotent and  $x\in\H\setminus\{0\}$. Let
$$k_{x}(T)=\limsup\limits_{\lambda\rightarrow 0}\frac{\ln\|(\lambda-T)^{-1}x\|}{\ln\|(\lambda-T)^{-1}\|}.$$
Usually, we briefly write $k_x$ instead of $k_{x}(T)$  if there is no confusion possible.
Set $\Lambda(T)=\{k_x:x\ne 0\}$, and call it the {\em power set} of $T$.
\end{defn}

\begin{rem}\label{basic remark}
For an quasinilpotent operator $T\in \B(\H)$, we list some basic and often-used facts about $\Lambda(T)$.
\begin{enumerate}

\item From \[\frac{\|x\|}{|\lambda|+\|T\|}\leq\|(\lambda-T)^{-1}x\|\leq\|(\lambda-T)^{-1}\|\|x\|,~~\forall\lambda\ne 0,\]
it follows that $\Lambda(T)\subset [0,1]$.

\item Since $\lim\limits_{\lambda\rightarrow 0}\|(\lambda- T)^{-1}\|=+\infty$,
it follows that $k_{\alpha x}=k_x$ for all $\alpha\ne 0$ and $x\ne 0$.
Thus, $\Lambda(T)=\{k_x:\|x\|=1\}$.

\item
It is not hard to check that the power set is invariant under dilation, i.e., $\Lambda(\alpha T)=\Lambda(T)$ for all $\alpha\neq 0$. In this sense, the power set could be viewed as a kind of ``{\em projective spectrum}".

\item
If $T=0$, then it is straightforward to check that $\Lambda(T)=\{1\}$.

\item $\Lambda(A)=\Lambda(T)$ if $A$ is similar to $T$ \cite[Proposition 5.2]{DY3}.

\end{enumerate}
\end{rem}
Recall that the celebrated hyperinvariant subspace problem asks: Does every bounded operator $T$ on $\H$ have a non-trivial closed hyperinvariant (i.e., invariant for all the operators commuting with $T$) subspace? This question is still open, especially for quasinilpotent operators in $\B(\H)$. We refer the readers to \cite{AE98,FJKP05,FP05, RR73} and the references therein for more information along this line.
The following result of Douglas and Yang established a link between the power set and the hyperinvariant subspace problem.

\begin{prop}\cite[Proposition 7.1, Corollary 7.2]{DY3}
Let $T\in \B(\H)$ be a quasinilpotent operator. For $\tau\in[0,1]$, write $M_\tau=\{x\in \H: x=0~{\it or}~k_{x}\leq\tau\}$. Then, $M_\tau$ is a linear subspace of $\H$, and $A(M_\tau)\subset M_\tau$ for every $A$ commuting with $T$.

In particular, when $\Lambda(T)$ contains two different points $\tau$ with closed $M_\tau$, $T$ has a nontrivial closed hyperinvariant subspace.

\end{prop}

Now we review known results about the calculations of the power set of quasinilpotent operators. In \cite{DY3}, Douglas and Yang proved that if $T\in \B(\H)$ with $T^n=0$ and $T^{n-1}\neq 0$, then $\Lambda(T)=\{\frac{j}{n}: j=1, 2, \cdots, n\}$. In \cite{JiLiu},
Ji and Liu showed that if $T$ is a strongly strictly cyclic quasinilpotent unilateral weighted shift, then $\Lambda(T)=\{1\}$.
Later, He and Zhu generalized Ji and Liu's above mentioned result by showing that if $T$ is either a strictly cyclic quasinilpotent unilateral weighted shift or  a strongly strictly cyclic quasinilpotent operator, then $\Lambda(T)=\{1\}$ \cite{HeZhu}. In contrast to the above discrete cases, in \cite{JiLiu}, Ji and Liu constructed a quasinilpotent operator $T\in \B(\H)$ with $\Lambda(T)=[0,1]$.
All of these results give a hint that the power set of a quasinilpotent operator seems always to be closed.

Comparing with the facts that $\sigma(A)$ is compact for every $A$ in $\B(\H)$ and each nonempty compact subset of $\bC$ is the spectrum of some operator in $\B(\H)$, one may naturally consider the following questions.

\begin{ques}\label{Question-closedness}
 Is $\Lambda(T)$ closed for each quasinilpotent operator $T\in \B(\H)$?
\end{ques}

\begin{ques}\label{Ques-not-closed}
 What kind of nonempty subset $\sigma$ of $[0,1]$ could be the power set of a quasinilpotent operator $T\in \B(\H)$?
\end{ques}

We next propose the notion of {\em right closed}.

\begin{defn}
A nonempty subset $X$ of $\bR$ is said to be right closed, if $\sup Y\in X$ for every nonempty bounded subset $Y$ of $X$.
\end{defn}

In this paper, we aim to investigate some fundamental properties of the power set. The first
goal of the present paper is to establish the followings.

\begin{theoremi}\label{A}
Suppose that $T\in\B(\H)$ is quasinilpotent. Then $\Lambda(T)$ is right closed.
\end{theoremi}

\begin{theoremi}\label{B}
Suppose that $\sigma$ is a right closed subset of $[0,1]$ containing $1$. Then there exists a quasinilpotent operator $T\in \B(\H)$ such that $\Lambda(T)=\sigma$.
\end{theoremi}

Note that $\sigma\doteq (\frac{1}{2}, 1]$ is a typical right closed but non-closed subset of $[0,1]$.
Thus, combining Theorem \ref{A} and Theorem \ref{B}, we give a complete answer to Question \ref{Question-closedness}.
Meanwhile, Theorem \ref{B} is also an answer to Question \ref{Ques-not-closed} under the mild assumption that $1\in \sigma$.
We suspect that this assumption can not be removed while every known calculable power set of the quasinilpotent operators always contains $1$.

In the second part of the paper, we focus on the power set of the Volterra operator on $L^2[0,1]$.
%Let $L^2[0,1]$ be the Lebesgue space on interval $[0,1]$ of measurable functions $f$ with norm $\|f\|=(\int_{[0,1]}|f(t)|^2\rd t)^{\frac{1}{2}}$.
Throughout this paper, let $V$ be the Volterra operator on $L^2[0, 1]$ defined by
\[Vf(t)=\int_0^tf(s)\rd s, ~t\in[0,1],~\forall f\in L^2[0,1].\]
 $V$ is a distinguished example of the quasinilpotent operator, hence has been extensively studied. A classical reference is \cite{GK70}. The invariant subspace lattice of $V$ coincides with $\{N_t:t\in[0,1]\}$, where \[N_t=\{f\in L^2[0,1]:f(s)=0~\textrm{a.e. in } [0,t]\},\] and hence $V$ is unicellular (cf. \cite{Dav88,RR73,Sar74}). About the power set of $V$, we could
assert the following.

\begin{theoremi}\label{C}
 $\Lambda(V)=(0,1]$.
\end{theoremi}

The key ingredient of Theorem \ref{C} is to show that $0\notin \Lambda(V)$. We will use some tools in complex analysis and some rather complicated estimates in place to obtain it.

\section{The power set is right closed}

The main goal of this section is to present the proof of Theorem \ref{A}.

\begin{proof}[\textup{\textbf{Proof of Theorem \ref{A}}}]
To show that $\Lambda(T)$ is right closed, it suffices to prove that $t\in\Lambda(T)$ whenever $\{t_n\}_{n=1}^\infty\subset\Lambda(T)$ is a strictly increasing sequence of positive numbers with the limit $t\in\bR$.
The strategy is to select a sequence of nonzero vectors $\{x_n\}_{n=1}^\infty$ inductively such that  $k_{x_n}=t_n$ and $k_x=t$, where $x=\sum\limits_{n=1}^\infty x_n\in \H$.
We will carry out it step by step.

\begin{enumerate}

%Step One==============
\item[\textsc{Step}] \textsc{1.}

Since $t_n\in\Lambda(T)$, we can choose a sequence $\{e_n\}_{n=1}^\infty$ of unit vectors with $k_{e_n}=t_n$, $\forall n$.
 Because of
\[k_{2^{n+2}e_n}=k_{e_n}=t_n<\frac{t_n+t_{n+1}}{2},\]
and
\[\underset{\lambda\to 0}{\lim}\|(\lambda-T)^{-1}\|=+\infty,\]
 we can pick a strictly decreasing sequence  $\{r_n\}_{n=1}^\infty$ in $(0,1)$ with $\lim\limits_{n\rightarrow\infty}r_n=0$ so that
\[\min\{\|(\lambda-T)^{-1}\|: 0\neq |\lambda|\leq r_1\}> 1\]
and
  \[\|(\lambda-T)^{-1}e_n\|<\frac{1}{2^{n+2}}\|(\lambda-T)^{-1}\|^{\frac{t_n+t_{n+1}}{2}},~~\forall~|\lambda|< r_n.\]

%Step Two==============
\item[\textsc{Step}] \textsc{2.}

We will find $\{x_n\}_{n=1}^\infty$ and $\{\lambda_n\}_{n=1}^\infty$ by induction as follows.

\begin{enumerate}
\item Induction part 1.
\begin{itemize}
\item[(1.1)] Set $s_1=\frac{1}{2}$ and $\|x_1\|=s_1e_1$, then $k_{x_1}=k_{e_1}=t_1>\frac{t_1}{2}$. So, one can find $\lambda_1$
with $0<|\lambda_1|<r_2$ such that
\[\|(\lambda_1-T)^{-1}x_1\|=\max\{\|(\lambda-T)^{-1}x_1\|:|\lambda|=|\lambda_1|\},\]
and
\[\ln\|(\lambda_1-T)^{-1}x_1\|>\frac{t_1}{2}\ln\|(\lambda_1-T)^{-1}\|.\]

\end{itemize}

\item Induction part 2.
 \begin{itemize}
\item[(2.1)] Since
\[\theta_1\doteq \min\{\frac{1}{2^{2+1}}\|(\lambda-T)^{-1}x_1\|: |\lambda|\in[|\lambda_1|,1]\}>0,\]
and
\[\max\{\|(\lambda-T)^{-1}e_2\|:|\lambda|\geq|\lambda_1|\}<+\infty,\]
 we can choose  $s_2$ in $(0, \frac{1}{2^2})$ so that $x_2=s_2e_2$ satisfies that
 \[\max\{\|(\lambda-T)^{-1}x_2\|:|\lambda|\geq|\lambda_1|\}
 \leq\theta_1.\]
\item[(2.2)] Because of $k_{x_2}=k_{e_2}=t_2>\frac{t_1+t_2}{2}$, we can
 pick $\lambda_2$ with $0<|\lambda_2|<\min\{|\lambda_1|, r_3\}$ so that the followings hold.
 \[\|(\lambda_2-T)^{-1}x_2\|=\max\{\|(\lambda-T)^{-1}x_2\|:|\lambda|=|\lambda_2|\},\]
 and
 \[\ln\|(\lambda_2-T)^{-1}x_2\|>\frac{t_1+t_2}{2}\ln\|(\lambda_2-T)^{-1}\|.\]
\end{itemize}

\item Induction part 3.
\begin{itemize}
\item[(3.1)] Since
\[\theta_2\doteq \min\{\frac{1}{2^{3+1}}\|(\lambda-T)^{-1}x_1\|, \frac{1}{2^{3+2}}\|(\lambda-T)^{-1}x_2\|:|\lambda|\in[|\lambda_2|,1]\}>0,\]
and
\[\max\{\|(\lambda-T)^{-1}e_3\|:|\lambda|\geq|\lambda_2|\}<+\infty,\]
we can choose  $s_3$ in $(0, \frac{1}{2^3})$ so that $x_3=s_3e_3$ satisfies that
 \[\max\{\|(\lambda-T)^{-1}x_3\|:|\lambda|\geq|\lambda_2|\}
 \leq\theta_2.\]
\item[(3.2)] Because of $k_{x_3}=k_{e_3}=t_3>\frac{t_2+t_3}{2}$, we can pick $\lambda_3$ with $0<|\lambda_3|<\min\{|\lambda_2|, r_4\}$ so that the followings hold.
    \[\|(\lambda_3-T)^{-1}x_3\|=\max\{\|(\lambda-T)^{-1}x_3\|:|\lambda|=|\lambda_3|\}\]
    and
    \[\ln\|(\lambda_3-T)^{-1}x_3\|>\frac{t_2+t_3}{2}\ln\|(\lambda_3-T)^{-1}\|.\]
\end{itemize}

\item Assume that we have completed the Induction part $j$, $1\leq j\leq k$. We will next get the
Induction part ($k+1$).

\begin{itemize}
\item[(($k+$1).1)] Since
\[\theta_k\doteq \min\{ \frac{1}{2^{k+1+j}}\|(\lambda-T)^{-1}x_j\|:|\lambda|\in[|\lambda_k|, 1], 1\leq j\leq k\}>0\]
and since \[\max\{\|(\lambda-T)^{-1}e_{k+1}\|:|\lambda|\geq|\lambda_k|\}<+\infty,\]
 We can choose  $s_{k+1}$ in $(0, \frac{1}{2^{k+1}})$ so that $x_{k+1}=s_{k+1}e_{k+1}$ satisfies that
 \[\max\{\|(\lambda-T)^{-1}x_{k+1}\|:|\lambda|\geq|\lambda_k|\}
 \leq\theta_k.\]
\item[(($k+$1).2)] Because of $k_{x_{k+1}}=k_{e_{k+1}}=t_{k+1}>\frac{t_k+t_{k+1}}{2}$, we can pick $\lambda_{k+1}$ with $0<|\lambda_{k+1}|<\min\{|\lambda_k|, r_{k+2}\}$ so that the followings hold.
    \[\|(\lambda_{k+1}-T)^{-1}x_{k+1}\|=\max\{\|(\lambda-T)^{-1}x_{k+1}\|:|\lambda|=|\lambda_{k+1}|\},\]
 and
\[\ln\|(\lambda_{k+1}-T)^{-1}x_{k+1}\|>\frac{t_k+t_{k+1}}{2}\ln\|(\lambda_{k+1}-T)^{-1}\|.\]
\end{itemize}

\end{enumerate}

%Step Three==============
\item[\textsc{Step}] \textsc{3.}

Note that
\[1=\frac{1}{2}+\sum_{k=2}^\infty \frac{1}{2^k}>\|x_1\|+\sum_{k=2}^\infty\|x_k\|> \|x_1\|-\sum_{k=2}^\infty\|x_k\|>\frac{1}{2}-\sum_{k=2}^\infty \frac{1}{2^k}\geq 0.\]
It follows that $x\doteq \sum\limits_{k=1}^\infty x_k\in \H$ with $\|x\|\in (0,1)$. We will next show that $k_x\leq t$.

For given $\lambda\ne 0$ with $|\lambda|<|\lambda_1|$, there exists a unique natural number $n=n(\lambda)$ such that $|\lambda|\in[|\lambda_{n+1}|, |\lambda_n|)$. Note that $\lim\limits_{\lambda\rightarrow 0} n(\lambda)=+\infty$.
\[\begin{array}{rl}\frac{\ln \|(\lambda-T)^{-1}x\|}{\ln\|(\lambda-T)^{-1}\|}&=\frac{\ln\|\sum\limits_{k=1}^\infty (\lambda-T)^{-1}x_k\|}{\ln\|(\lambda-T)^{-1}\|}\\
&\leq\frac{\ln[\sum\limits_{k=1}^{n+1}\|(\lambda-T)^{-1}x_k\|+\sum\limits_{k=n+2}^\infty\|(\lambda-T)^{-1}x_k\|]}{\ln\|(\lambda-T)^{-1}\|}\\
&\leq\frac{\ln[\sum\limits_{k=1}^{n+1}\|(\lambda-T)^{-1}x_k\|+\sum\limits_{k=n+2}^\infty\theta_{k-1}]}{\ln\|(\lambda-T)^{-1}\|}\\
&\leq\frac{\ln[\sum\limits_{k=1}^{n+1}\|(\lambda-T)^{-1}x_k\|+\sum\limits_{k=n+2}^\infty\frac{1}{2^{k+n+1}}\|(\lambda-T)^{-1}x_{n+1}\|]}{\ln\|(\lambda-T)^{-1}\|}\\
&\leq\frac{\ln[\sum\limits_{k=1}^{n+1}s_k\|(\lambda-T)^{-1}e_k\|+\sum\limits_{k=n+2}^\infty\frac{s_{n+1}}{2^{k+n+1}}\|(\lambda-T)^{-1}e_{n+1}\|]}{\ln\|(\lambda-T)^{-1}\|}\\
&\leq\frac{\ln[\sum\limits_{k=1}^{n+1}\frac{s_k}{2^{k+2}}\|(\lambda-T)^{-1}\|^{\frac{t_k+t_{k+1}}{2}}+\sum\limits_{k=n+2}^\infty\frac{s_{n+1}}{2^{k+n+1}}\frac{1}{2^{n+1+2}}\|(\lambda-T)^{-1}\|^{\frac{t_{n+1}+t_{n+2}}{2}}]}{\ln\|(\lambda-T)^{-1}\|}\\
&\leq\frac{\ln[\sum\limits_{k=1}^{n+1}\frac{s_k}{2^{k+2}}\|(\lambda-T)^{-1}\|^t+\frac{s_{n+1}}{2^{2n+2}}\frac{1}{2^{n+1+2}}\|(\lambda-T)^{-1}\|^t]}{\ln\|(\lambda-T)^{-1}\|}\\
&\leq\frac{\ln[\sum\limits_{k=1}^{n+1}\frac{1}{2^{k+2}}\|(\lambda-T)^{-1}\|^t+\frac{1}{2^{2n+2}}\frac{1}{2}\|(\lambda-T)^{-1}\|^t]}{\ln\|(\lambda-T)^{-1}\|}\\
&\leq t.\end{array}\]

So, $k_x\leq t$.

%Step Four==============
\item[\textsc{Step}] \textsc{4.}

We finally will show $k_x=t$.
\[\begin{array}{rl}\frac{\ln \|(\lambda_{n+1}-T)^{-1}x\|}{\ln\|(\lambda_{n+1}-T)^{-1}\|}&=\frac{\ln\|\sum\limits_{k=1}^\infty (\lambda_{n+1}-T)^{-1}x_k\|}{\ln\|(\lambda_{n+1}-T)^{-1}\|}\\
&\geq\frac{\ln[\|(\lambda_{n+1}-T)^{-1}x_{n+1}\|-(\|\sum\limits_{k=1}^{n}\|(\lambda_{n+1}-T)^{-1}x_k\|+\sum\limits_{k=n+2}^\infty\|(\lambda_{n+1}-T)^{-1}x_k\|)]}{\ln\|(\lambda_{n+1}-T)^{-1}\|}\\
&\geq\frac{\ln[\|(\lambda_{n+1}-T)^{-1}x_{n+1}\|-(\|\sum\limits_{k=1}^{n}\|(\lambda_{n+1}-T)^{-1}x_k\|+\sum\limits_{k=n+2}^\infty\theta_{k-1})]}{\ln\|(\lambda_{n+1}-T)^{-1}\|}\\
&\geq\frac{\ln[\|(\lambda_{n+1}-T)^{-1}x_{n+1}\|-(\sum\limits_{k=1}^n\|(\lambda_{n+1}-T)^{-1}x_k\|
        +\sum\limits_{k=n+2}^\infty\frac{1}{2^{k+n+1}}\|(\lambda_{n+1}-T)^{-1}x_{n+1}\|)]}{\ln\|(\lambda_{n+1}-T)^{-1}\|}\\
&=\frac{\ln[(1-\frac{1}{2^{2n+2}})\|(\lambda_{n+1}-T)^{-1}x_{n+1}\|-\sum\limits_{k=1}^ns_k\|(\lambda_{n+1}-T)^{-1}e_k\|]}{\ln\|(\lambda_{n+1}-T)^{-1}\|}\\
&\geq\frac{\ln[(1-\frac{1}{2^{2n+2}})\|(\lambda_{n+1}-T)^{-1}\|^{\frac{t_n+t_{n+1}}{2}}-\sum\limits_{k=1}^n\frac{s_k}{2^{k+2}}\|(\lambda_{n+1}-T)^{-1}\|^{\frac{t_k+t_{k+1}}{2}}]}{\ln\|(\lambda_{n+1}-T)^{-1}\|}\\
&\geq\frac{\ln[(1-\frac{1}{2^{2n+2}})\|(\lambda_{n+1}-T)^{-1}\|^{\frac{t_n+t_{n+1}}{2}}-\sum\limits_{k=1}^n\frac{s_k}{2^{k+2}}\|(\lambda_{n+1}-T)^{-1}\|^{\frac{t_n+t_{n+1}}{2}}]}{\ln\|(\lambda_{n+1}-T)^{-1}\|}\\
&\geq\frac{\ln[(1-\frac{1}{2^{2n+2}})\|(\lambda_{n+1}-T)^{-1}\|^{\frac{t_n+t_{n+1}}{2}}-\sum\limits_{k=1}^n\frac{1}{2^{k+2}}\|(\lambda_{n+1}-T)^{-1}\|^{\frac{t_n+t_{n+1}}{2}}]}{\ln\|(\lambda_{n+1}-T)^{-1}\|}\\
&\geq\frac{t_n+t_{n+1}}{2}+\frac{\ln\frac{11}{16}}{\ln\|(\lambda_{n+1}-T)^{-1}\|}.\end{array}\]
So, $k_x\geq t$. And hence $k_x=t$.
\end{enumerate}

The proof is completed.
\end{proof}

\section{Realizability of right closed subset of $[0,1]$}

Thanks to Theorem \ref{A}, we now know that $\Lambda(T)$ is a nonempty right closed subset of $[0,1]$ for a quasinilpotent operator $T\in \B(\H)$.
Then it is quite natural to ask that whether each nonempty right closed subset of $[0,1]$ could be realizable?  More precisely, for a given nonempty right closed subset $\sigma$ of $[0,1]$, is it possible to find a quasinilpotent $T\in \B(\H)$ such that
$\Lambda(T)=\sigma$? Under the mild assumption that $1\in \sigma$, Theorem \ref{B} provides an positive answer to this question.
The proof of Theorem \ref{B} will be carried out in this section.

Throughout this section, let $A$ be the weighted shift with weight sequence $\{\frac{1}{n}\}_{n=1}^\infty$, namely, there exists an orthonormal basis $\{e_n\}_{n=0}^\infty$ of $\H$ such that $A$ admits presenting matrix as
$$
\left[\begin{array}{ccccc}
0& & & & \\
1&0& & & \\
 &\frac{1}{2}&0& & \\
 & &\frac{1}{3}&0& \\
 & & &\ddots&\ddots
 \end{array}\right].
$$

%As noticed in \cite[Example 2]{JiLiu}, $A$ is actually the Volterra operator on Hardy space $H^2(\bD)$, defined by
%\[(Vf)(z)=\int_0^z f(t)\rd t,~~f\in H^2(\bD).\]\cite[Theorem 1.3]{LY}

By \cite[Corollary 3.12]{JiLiu}, $\Lambda(A)=\{1\}$. Note that $\Lambda(0)=\{1\}$, hence by Remark \ref{basic remark}(3),
we have
\begin{lem}\label{rA}
 $\Lambda(rA)=\{1\}$, $\forall r\geq 0$.
\end{lem}
We will use the above lemma later on.

\begin{lem}\label{lower and upper bound of norm}
Given $t>0,r>0$,
\[\frac{\sqrt{6}}{\pi r}(e^{rt}-1)\leq \|(\frac{1}{t}-rA)^{-1}\|\leq te^{rt}.\]
\end{lem}

\begin{proof}
 Set $x=\sum\limits_{k=0}^\infty\frac{1}{k+1}e_k$. Then
\[\|x\|=\sqrt{\sum\limits_{k=0}^\infty\frac{1}{(k+1)^2}}=\sqrt{\frac{\pi^2}{6}}.\]
Hence, as
\[(\frac{1}{t}-rA)^{-1}=\sum\limits_{n=0}^\infty t^{n+1}(rA)^n,\]
 we have
\begin{align*}
\|(\frac{1}{t}-rA)^{-1}\|
%&=\|(\frac{1}{t}-(rA)^\ast)^{-1}\|\geq\frac{1}{\|x\|}\|(\frac{1}{t}-(rA)^\ast)^{-1}x\|\geq \frac{1}{\|x\|}|(e_0,(\frac{1}{t}-(rA)^\ast)^{-1}x)|\\
&\geq \frac{1}{\|x\|}|((\frac{1}{t}-rA)^{-1}e_0,x)|\\
&=\frac{1}{\|x\|}|(\sum\limits_{n=0}^\infty t^{n+1}(rA)^ne_0,\sum\limits_{k=0}^\infty \frac{e_k}{k+1})|\\
&=\frac{1}{\|x\|}|(\sum\limits_{n=0}^\infty\frac{t^{n+1}r^n}{n!}e_n,\sum\limits_{k=0}^\infty \frac{e_k}{k+1})|\\
&=\frac{1}{\|x\|}\sum\limits_{n=0}^\infty \frac{t^{n+1}r^{n}}{(n+1)!}\\
&=\frac{1}{\|x\|r}(e^{rt}-1).
\end{align*}

On the other side, and notice that
\[\|(rA)^n\|=\frac{r^n}{n!},~~n\in \bN.\]
So
\[\|(\frac{1}{t}-rA)^{-1}\|\leq \sum\limits_{n=0}^\infty\frac{t^{n+1}r^{n}}{n!}=te^{rt}.\]

\end{proof}

\begin{prop}\label{arg}
$\|(z-rA)^{-1}\|=\|(|z|-rA)^{-1}\|$, $\forall z\ne 0$, $r\geq 0$.
%$\|(z-rA)^{-1}e_0\|=\|(|z|-rA)^{-1}e_0\|$,
\end{prop}

\begin{proof}
Write $z=|z|e^{i\theta}$, where $\theta\in [0,2\pi)$. Set
\[W(\sum_{n=0}^\infty \alpha_n e_n)=\sum_{n=0}^\infty \alpha_n e^{in\theta}e_n,\]
where $\sum\limits_{n=0}^\infty \alpha_n e_n\in \H.$
Then $W$ is a unitary and
\[e^{in\theta}A^n=WA^nW^*,~~n\in \bN.\]
Therefore,
\[W(z-rA)^{-1}W^*=\sum_{n=0}^\infty \frac{W(rA)^nW^*}{z^{n+1}}=\frac{1}{z}\sum_{n=0}^\infty\frac{(e^{ i\theta})^n(rA)^n}{|z|^n(e^{ i\theta})^n}=\frac{1}{z}\sum_{n=0}^\infty\frac{(rA)^n}{|z|^n}.\]
Then
\[\|(z-rA)^{-1}\|=\|W(z-rA)^{-1}W^*\|=\frac{1}{|z|}\|\sum_{n=0}^\infty\frac{(rA)^n}{|z|^n}\|=
\|\sum_{n=0}^\infty\frac{(rA)^n}{|z|^{n+1}}\|=\|(|z|-rA)^{-1}\|.\]

%Since $We_0=e_0$,
%\begin{align*}
%\|(z-rA)^{-1}e_0\|&=\|(z-rA)^{-1}We_0\|\\
%&=\|W^*(z-rA)^{-1}We_0\|\\
%&=\frac{1}{|z|}\|\sum_{n=0}^\infty\frac{(rA)^n}{|z|^n}e_0\|\\
%&=\|\sum_{n=0}^\infty\frac{(rA)^n}{|z|^{n+1}}e_0\|\\
%&=\|(|z|-rA)^{-1}e_0\|.
%\end{align*}

\end{proof}

\begin{lem}\label{increasing function}
If $r_1>r_2\geq 0$, then $\|(z-r_1A)^{-1}\|\geq \|(z-r_2A)^{-1}\|$, $\forall z\ne 0$.
\end{lem}

\begin{proof}
Suppose that $a>b\geq 0$ satisfying
\[\max\limits_{|w|=a}\|(1-wA)^{-1}\|<\max\limits_{|w|=b}\|(1-wA)^{-1}\|.\]
By Proposition \ref{arg} and the Hahn-Banach theorem, there exists $\phi\in (\B(\H))^*$, the conjugate space of $\B(\H)$, such that $\|\phi\|=1$ and
\[|\phi((1-bA)^{-1})|=\|(1-bA)^{-1}\|=\max\limits_{|w|=b}\|(1-wA)^{-1}\|.\]
In particular,
\[|\phi((1-bA)^{-1})|>\max\limits_{|w|=a}\|(1-wA)^{-1}\|\geq \max\limits_{|w|=a}|\phi((1-wA)^{-1})|.\]
Next consider the entire function
\[F(w)\doteq \phi((1-wA)^{-1})=\sum_{n=0}^\infty \phi(A^n)w^n.\]
Then by the maximum modulus principle,
\[|\phi((1-bA)^{-1})|=|F(b)|\leq \max\limits_{|w|=a}|F(w)|=\max\limits_{|w|=a}|\phi((1-wA)^{-1})|.\]
This contradiction yields that $\max\limits_{|w|=r}\|(1-wA)^{-1}\|$ is increasing with $r$. By Proposition \ref{arg},
\[\max\limits_{|w|=r}\|(1-wA)^{-1}\|=\|(1-rA)^{-1}\|,~~\forall r\geq 0.\]
Therefore, \[\|(z-r_1A)^{-1}\|=\frac{1}{|z|}\|(1-\frac{r_1}{|z|}A)^{-1}\|\geq \frac{1}{|z|}\|(1-\frac{r_2}{|z|}A)^{-1}\|=\|(z-r_2A)^{-1}\|.\]
\end{proof}

\begin{lem}\label{quotient}
$\forall r\geq 0$,
\[\lim\limits_{z\rightarrow 0}\frac{\ln \|(z-rA)^{-1}\|}{\ln \|(z-A)^{-1}\|}=r.\]
\end{lem}
\begin{proof}
Since $rA$ is quasinilpotent, it follows from the spectral mapping theorem that
\[\|(z-rA)^{-1}\|\geq \frac{1}{|z|}>10\]
and hence
\[\ln \|(z-rA)^{-1}\|>\ln 10,~~\textup{if}~|z|<\frac{1}{10}.\]

Then, according to Lemma \ref{lower and upper bound of norm},
\[\frac{\ln \|(|z|-rA)^{-1}\|}{\ln [\frac{1}{|z|}e^{\frac{1}{|z|}}]}\leq\frac{\ln \|(|z|-rA)^{-1}\|}{\ln \|(|z|-A)^{-1}\|}\leq \frac{\ln \|(|z|-rA)^{-1}\|}{\ln [\frac{\sqrt{6}}{\pi}(e^{\frac{1}{|z|}}-1)]},~~\textup{if}~|z|<\frac{1}{10}.\]
\begin{itemize}
\item For $r=0$, \[\frac{\ln \frac{1}{|z|}}{\ln [\frac{1}{|z|}e^{\frac{1}{|z|}}]}\leq\frac{\ln \|(|z|-rA)^{-1}\|}{\ln \|(|z|-A)^{-1}\|}\leq \frac{\ln \frac{1}{|z|}}{\ln [\frac{\sqrt{6}}{\pi}(e^{\frac{1}{|z|}}-1)]},~~\textup{if}~|z|<\frac{1}{10}.\]
Therefore, by Proposition \ref{arg},
\[\lim\limits_{z\rightarrow 0}\frac{\ln \|(z-rA)^{-1}\|}{\ln \|(z-A)^{-1}\|}=\lim\limits_{|z|\rightarrow 0}\frac{\ln \|(|z|-rA)^{-1}\|}{\ln \|(|z|-A)^{-1}\|}=0.\]
\item For $r>0$, by Lemma \ref{lower and upper bound of norm}, for $|z|<\frac{1}{10}$,
\[\frac{\ln[\frac{\sqrt{6}}{r\pi}(e^{r\frac{1}{|z|}}-1)]}{\ln [\frac{1}{|z|}e^{\frac{1}{|z|}}]}\leq
\frac{\ln\|(|z|-rA)^{-1}\|}{\ln [\frac{1}{|z|}e^{\frac{1}{|z|}}]}\leq
\frac{\ln \|(|z|-rA)^{-1}\|}{\ln \|(|z|-A)^{-1}\|}
\leq \frac{\ln \|(|z|-rA)^{-1}\|}{\ln[\frac{\sqrt{6}}{\pi}(e^{\frac{1}{|z|}}-1)]}
\leq \frac{\ln [\frac{1}{|z|}e^{r\frac{1}{|z|}}]}{\ln[\frac{\sqrt{6}}{\pi}(e^{\frac{1}{|z|}}-1)]}.\]
Therefore, again by Proposition \ref{arg},
\[\lim\limits_{z\rightarrow 0}\frac{\ln \|(z-rA)^{-1}\|}{\ln \|(z-A)^{-1}\|}=\lim\limits_{|z|\rightarrow 0}\frac{\ln \|(|z|-rA)^{-1}\|}{\ln \|(|z|-A)^{-1}\|}=r.\]
\end{itemize}

\end{proof}

We turn to introduce the notion of {\em right closure}.

\begin{defn}
Given a nonempty subset $K$ of $\bR$, define\\
\hspace*{60pt} $\lceil K \rceil=\{\sup\sigma: \emptyset\ne\sigma\subset K, \sigma \textrm{ is bounded}\}$,\\
and call it the {\em right closure} of $K$. It is clear that the right closure of $K$ is the smallest right closed set containing $K$.
\end{defn}

\begin{lem}\label{sequence}
If $\sigma$ is a nonempty right closed subset of $\bR$, then there exists a sequence $\{r_n:n\geq 1\}$ such that whose right closure is $\sigma$.
\end{lem}

\begin{proof}
Without loss of generality, we may assume that $\sigma$ is not a finite set.
Let
\[\sigma_0=\{t\in\sigma: \exists\delta>0 \textrm{ such that } \sigma\cap(t-\delta, t))=\emptyset\}.\] Then, it is routine to see that $\sigma_0$ is at most countable. Pick a (at most) countable dense subset $\sigma_1$ of $\sigma$. List $\sigma_0\cup\sigma_1$ as $\{r_n:n\geq 1\}$. Then $\sigma$ is the right closure of $\{r_n:n\geq 1\}$.
\end{proof}

Now we are ready to give the proof of Theorem \ref{B}.

\begin{proof}[\textup{\textbf{Proof of Theorem \ref{B}}}]
By Lemma \ref{sequence}, pick a sequence $\{r_k:k\geq 1\}$ of $\sigma$ so that $r_1=1$ and that $\sigma$ is the right closure of $\{r_k:k\geq 1\}$. Set
$T=\doplus\limits_{k=1}^\infty r_kA$
which is acting on $\K=\doplus\limits_{k=1}^\infty \H$.
Since
\[\|T^n\|=\|A^n\|,~~\forall n\in \mathbb{N},\]
it follows that $T$ is quasinilpotent. Next, we will show that $\Lambda(T)=\sigma$.
\begin{itemize}
\item For $z\neq 0$,
since $r_1=1$,
\[\|(z-T)^{-1}\|=\underset{k\in \bN}{\sup}\|(z-r_kA)^{-1}\|\geq \|(z-A)^{-1}\|;\]
on the other hand, by Lemma \ref{increasing function},
\[\|(z-r_kA)^{-1}\|\leq \|(z-A)^{-1}\|,~~\forall k\in \mathbb{N}.\]
Hence,
\[\|(z-T)^{-1}\|=\|(z-A)^{-1}\|,~~\forall z\neq 0.\]

\item Given $x=\doplus\limits_{k=1}^\infty x_k\in \K$ with $\|x\|=1$.
Set $l=\sup\{r_k:x_k\ne 0\}$. On the one hand,
\begin{align*}
k_{x}(T)& =\limsup\limits_{z\rightarrow 0}\frac{\ln\|(z-T)^{-1}x\|}{\ln\|(z-T)^{-1}\|}\\
	&= \limsup\limits_{z\rightarrow 0}\frac{\frac{1}{2}\ln(\sum_{k=1}^\infty \|(z-r_kA)^{-1}x_k\|^2)}{\ln\|(z-T)^{-1}\|} \\
	&\leq  \limsup\limits_{z\rightarrow 0}\frac{\frac{1}{2}\ln(\sum_{k=1}^\infty \|(z-r_kA)^{-1}\|^2\|x_k\|^2)}{\ln\|(z-T)^{-1}\|} \\
	&\overset{\textup{Lemma~\ref{increasing function}}}{\leq} \limsup\limits_{z\rightarrow 0}\frac{\frac{1}{2}\ln( \|(z-lA)^{-1}\|^2)}{\ln\|(z-A)^{-1}\|} \\
&\overset{\textup{Lemma~\ref{quotient}}}{=}l.
\end{align*}
On the other hand,
fix a $k\in \bN$, such that $x_k\neq 0$. By Lemma \ref{rA}, $\Lambda(r_kA)=\{1\}$. Then
\begin{align*}
k_{x}(T)& =\limsup\limits_{z\rightarrow 0}\frac{\ln\|(z-T)^{-1}x\|}{\ln\|(z-T)^{-1}\|}\\
	&\geq  \limsup\limits_{z\rightarrow 0}\frac{\ln \|(z-r_kA)^{-1}x_k\|}{\ln\|(z-r_kA)^{-1}\|}\underset{z\to 0}{\lim}\frac{\ln\|(z-r_kA)^{-1}\|}{\ln\|(z-A)^{-1}\|} \\
&\overset{\textup{Lemma~\ref{quotient}}}{=}k_{x_k}(r_kA)\cdot r_k\\
&=r_k.
\end{align*}
Thus, $k_{x}(T)\geq l$. Therefore, $k_{x}(T)=l$.
Since $\sigma$ is right closed, it is clearly that

\[\Lambda(T)=\{k_{x}(T):\|x\|=1\}\subset \sigma.\]

\item Fix $k^*\in \bN$. Pick $x=\doplus\limits_{k=1}^\infty x_k\in \K$ with $x_{k}\neq 0$ $\iff$ $k=k^*$. It turns out that $k_{x}(T)=r_{k^*}$.
Hence, $\Lambda(T)\supset \{r_k:k\geq 1\}$. From Theorem \ref{A}, $\Lambda(T)$ is right closed. Thus
$\Lambda(T)\supset \sigma$, as $\sigma$ is the right closure of $\{r_k:k\geq 1\}$.

\end{itemize}

In summary, $\Lambda(T)=\sigma$.
\end{proof}

An immediate consequence of Theorem \ref{B} is that the power set of a quasinilpotent operator can be the Cantor set, and also can be $(\frac{1}{2},1]$ which is obvious non-closed. The following result is also a by-product of Theorem \ref{B}.

\begin{cor}\label{oplus no}
There exists a sequence of quasinilpotent operators $A_k\in \B(\H)$ such that $\doplus\limits_{k=1}^\infty A_k$
is a quasinilpotent operator in $\B(\K)$ and \[\Lambda(\doplus\limits_{k=1}^\infty A_k)\supsetneq \overline{\cup_{k=1}^\infty \Lambda(A_k)},\]
 where $\overline{\cup_{k=1}^\infty \Lambda(A_k)}$ is the closure of $\cup_{k=1}^\infty \Lambda(A_k)$, and $\K=\doplus\limits_{k=1}^\infty \mathcal{H}$.
\end{cor}

\begin{proof}
List all the rational numbers in $(0,1]$ as a sequence $\{r_k:k\geq 1\}$. Then the right closure of $\{r_k:k\geq 1\}$
 is $(0,1]$.
Set $A_k=r_kA$, $k\in \bN$. Then $\doplus\limits_{k=1}^\infty A_k\in \B(\K)$ is  quasinilpotent.
By Lemma \ref{rA}, $\Lambda(A_k)=\{1\}$.
After revisiting the proof of Theorem \ref{B}, we obtain that $\Lambda(\doplus\limits_{k=1}^\infty A_k)=(0,1]$.
The proof is now completed.
\end{proof}

We end this section with an interesting question.

\begin{ques}
Whether or not $1$ is always in the power set of a quasinilpotent operator?
\end{ques}

\section{$\Lambda(V)=(0,1]$}

The main goal of this section is to show that
$\Lambda(V)=(0,1]$.  To obtain this, we need several auxiliary lemmas.
The following lemma is a corollary of the Phragmen-Lindel$\ddot{\rm o}$f theorem (see \cite[Page 139]{Con}).

\begin{lem}\label{lem PL}
Let $a\geq \frac{1}{2}$ and put
$$
G=\{z:|\textrm{arg }z|<\frac{\pi}{2a}\}.
$$
Suppose that $\varphi$ is analytic on $G$ and there is a constant $M$ such that $\limsup\limits_{z\rightarrow w}|\varphi(z)|\leq M$ for all $w$ in $\partial G$. If there are positive constants $P$, $R$ and $b<a$ such that
$$
|\varphi(z)|\leq P e^{(|z|^b)}
$$
for all $z$ with $|z|\geq R$, then $|\varphi(z)|\leq M$, $\forall z\in G$.
\end{lem}

The following lemma is an immediately consequence of Lemma \ref{lem PL}.

\begin{lem}\label{lem quarter}
For $\alpha\in[-\pi,\pi-\frac{\pi}{4}]$,  set
$$
G_\alpha=\{z:\textrm{arg }z\in(\alpha, \alpha+\frac{\pi}{2})\}.
$$
Suppose that $\varphi$ is analytic on $G_\alpha$ and continuous on $\overline{G_\alpha}$, the closure of $G_\alpha$, and suppose that there is a positive constant $M$ such that
$\sup\limits_{z\in\partial G_\alpha}|\varphi(z)|\leq M$. If there is a positive constant $P$ such that
$$
|\varphi(z)|\leq P e^{|z|}
$$
for all $z\in G_\alpha$, then $|\varphi(z)|\leq M$, $\forall z\in G_\alpha$.
\end{lem}

 \begin{proof}
 Set $a=2$. Let
 \[G=\{z: |z|<\frac{\pi}{2a}\}=\{z: |z|<\frac{\pi}{4}\},~~\tau(z)=ze^{i(\alpha+\frac{\pi}{4})}.\]
  Then $\tau$ is an analytic homeomorphism from $G$ onto $G_\alpha$. Pick $b=1$. Let $\psi=\varphi\circ\tau$. Then $\psi$ is analytic on $G$, and $|\psi(z)|\leq Pe^{|z|}$. By Lemma \ref{lem PL},
  \[\sup\{|\psi(z)|:z\in G\}\leq M.\]
   So,
  \[\sup\{|\varphi(z)|:z\in G_\alpha\}\leq M.\]
 \end{proof}

\begin{lem}\label{we need}
Let $\varphi$ is an entire function. Suppose that there is a positive constant $P$ such that
$$
|\varphi(z)|\leq P e^{|z|},
$$
for all $z\in\bC$. If $$\max\{\limsup\limits_{r\rightarrow+\infty}|\varphi(r)|,\, \limsup\limits_{r\rightarrow+\infty}|\varphi(-r)|,\, \limsup\limits_{r\rightarrow+\infty}|\varphi(ri)|,\, \limsup\limits_{r\rightarrow+\infty}|\varphi(-ri)|\}<+\infty,$$
 then $\varphi$ is a constant.
\end{lem}

\begin{proof}

It follows from Lemma \ref{lem quarter} that $\varphi$ is bounded. Hence it is a constant by  Liouville's theorem.

\end{proof}

For $f\in C[0,1]$, and $u\in(0,1]$, set the function
\[\Phi_{f,u}(z)=\int_0^uf(s)e^{(u-s)z}\rd{s},~~ \forall z\in\bC.\] It is standard to check that $\Phi_{f,u}(z)$ is an entire function on
$\bC$. Given $f\in C[0,1]$ and $u\in[0,1]$, set
\[P_{f,u}=\max\{|f(s)|:0\leq s\leq u\}.\]

\begin{lem}\label{up-lim-Phi}
Let $f\in C_{\mathbb{R}}[0,1]$ and $u\in(0,1]$ such that $P_{f,u}\ne 0$. Then $\limsup\limits_{\rho\rightarrow+\infty}|\Phi_{f,u}(\rho)|=+\infty$.
In particular, there is a sequence $\{\rho_k\}_{k=1}^\infty$ of positive numbers with $\lim\limits_{k\rightarrow\infty}\rho_k=+\infty$ such that $\{|\Phi_{f, u}(\rho_k)|\}_{k=1}^\infty$ is increasing and $\lim\limits_{k\rightarrow \infty}|\Phi_{f, u}(\rho_k)|=+\infty$.
\end{lem}

\begin{proof}
 It is not difficult to see that the following statements hold.
\begin{itemize}
\item[(1) ] $|\Phi_{f,u}(z)|\leq P_{f,u} e^{|z|}$, $\forall z\in \bC$.
\item[(2) ] $|\Phi_{f, u}(\rho i)|\leq P_{f,u}$, $\forall \rho\in\bR$.
\item[(3) ] $\lim\limits_{\rho\rightarrow+\infty}\Phi_{f,u}(-\rho)=0$.
\end{itemize}
Assume that $\limsup\limits_{\rho\rightarrow+\infty}|\Phi_{f,u}(\rho)|<+\infty$. It follows from the above-mentioned statement (3) and Lemma \ref{we need} that $\Phi_{f,u}(z)=0$, $\forall z\in\bC$.
So,
$$
\psi(z)\doteq\int_0^uf(s)e^{-zs}\rd s=\Phi_{f,u}(z) e^{-uz}=0.
$$
And then the $n$-th derivative of $\psi(z)$ equals to 0, $\forall z\in\bC$.
Notice that
\[\psi^{(n)}(z)=\int_0^u (-1)^ns^nf(s) e^{-zs}\rd s.\]
 Choose $z=0$. Then
$$
\int_0^us^nf(s)\rd s=0,~\forall n\geq 0.$$ Thus,
\[\int_0^u g(s)f(s)\rd s=0\]
for all polynomials. Then by the Weierstrass approximation theorem,  $f(s)=0$, $\forall s\in[0,u]$, and hence $P_{f,u}=0$. This contradiction shows that $\limsup\limits_{\rho\rightarrow+\infty}|\Phi_{f,u}(\rho)|=+\infty$.
\end{proof}

\begin{lem}\label{Phi-e-k}
Given $f\in C_{\mathbb{R}}[0,1]$, $u\in(0,1]$ and a positive constant $\rho$. If $f(u)\ne 0$, then there is a positive constant $d$ and a sequence $\{\rho_k\}_{k=1}^\infty$ with $\lim\limits_{k\rightarrow\infty}\rho_k=+\infty$ such that $$|\Phi_{f, u}(\rho_k)|>\rho e^{d\rho_k},$$
 $\forall k\geq 1$.
\end{lem}

\begin{proof}
Because $f$ is continuous and that $f(u)\ne 0$, we can pick $d\in(0, u)$ so that $f(u-d)\ne 0$. Write $v=u-d$. Then $P_{f,v}\ne 0$.  Pick Set $C=\rho+\|f\|_\infty$, where $\|f\|_\infty\doteq \underset{t\in [0,1]}{\max}|f(t)|$. By Lemma \ref{up-lim-Phi}, there is a sequence $\{\rho_k\}_{k=1}^\infty$ of positive numbers with $\lim\limits_{k\rightarrow\infty}\rho_k=+\infty$ such that $|\Phi_{f,v}(\rho_k)|>C$, $\forall k\geq 1$. Notice that
\begin{align*}
|\Phi_{f,u}(\rho_k)|&=|\int_0^uf(s)e^{(u-s)\rho_k}\rd s|\\
&=|\int_0^vf(s)e^{(u-s)\rho_k}\rd s+\int_v^uf(s)e^{(u-s)\rho_k}\rd s|\\
&\geq |\int_0^vf(s)e^{(u-s)\rho_k}\rd s|-|\int_v^uf(s)e^{(u-s)\rho_k}\rd s|.
\end{align*}

Notice that
$$
|\int_0^vf(s)e^{(u-s)\rho_k}\rd s|=|\int_0^vf(s)e^{(v-s)\rho_k}\rd se^{(u-v)\rho_k}|=|\Phi_{f,v}(\rho_k)|e^{d\rho_k}> Ce^{d\rho_k};
$$
and
$$
|\int_v^uf(s)e^{(u-s)\rho_k}\rd s|\leq\|f\|_\infty e^{(u-v)\rho_k}=\|f\|_\infty e^{d\rho_k}.
$$
So,
\[|\Phi_{f, u}(\rho_k)|>(C-\|f\|_\infty)e^{d\rho_k}=\rho e^{d\rho_k}.\]

\end{proof}

The following lemma is the complex version of Lemma \ref{Phi-e-k} and therefore the proof is omitted.
\begin{lem}\label{c-Phi-e-k}
Given $f\in C[0,1]$, $u\in(0,1]$ and a positive constant $\rho$. If $f(u)\ne 0$, then there is a positive constant $d$ and a sequence $\{\rho_k\}_{k=1}^\infty$ with $\lim\limits_{k\rightarrow\infty}\rho_k=+\infty$ such that $$|\Phi_{f, u}(\rho_k)|>\rho e^{d\rho_k},~\forall k\geq 1.$$
\end{lem}

\begin{lem}\label{commute}
Suppose that $T\in \B(\H)$ is quasinilpotent, and suppose that $A\in \B(\H)$ commutes with $T$. Then $k_{Ax}(T)\leq k_x(T)$ if $Ax\ne 0$, where $x\in \H$.
\end{lem}

\begin{proof}
Since $TA=AT$, it follows that
\[(\lambda-T)^{-1}A=A(\lambda-T)^{-1}\]
 and hence from the fact that $\underset{\lambda \to 0}{\lim}\|(\lambda-T)^{-1}\|=+\infty$,
\[k_{Ax}(T)=\limsup\limits_{\lambda\rightarrow 0}\frac{\ln\|A(\lambda-T)^{-1}x\|}{\ln\|(\lambda-T)^{-1}\|}\leq \limsup\limits_{\lambda\rightarrow 0}\frac{\ln(\|A\|+\ln\|(\lambda-T)^{-1}x\|)}{\ln\|(\lambda-T)^{-1}\|}\leq  k_{x}(T).\]

\end{proof}

As $V$ is injective, for $0\neq f\in L^2[0,1]$, Lemma \ref{commute} implies that $\{k_{V^nf}\}_{n=0}^{\infty}$ is a decreasing sequence. In particular, $k_{Vf}$ indeed plays an important role in the remaining part of this section.
Now we recall the following nice formulation of $(\lambda-V)^{-1}Vf$.

\begin{lem}\label{metric at Vf}
For $f\in L^2[0,1]$, $\lambda\neq 0$,
\[(\lambda-V)^{-1}Vf(t)=\frac{1}{\lambda}\int_0^te^{\frac{t-s}{\lambda}}f(s)\rd s,~~t\in [0,1].\]
\end{lem}

\begin{proof}
It is well known that
\[V^nf(t)=\int_0^t \frac{(t-s)^{n-1}}{(n-1)!}f(s)\rd s,~~n\geq 1.\]
Hence for $\lambda\neq 0$,
\[(\lambda-V)^{-1}Vf(t)=\sum_{n=0}^\infty \frac{V^{n+1}f}{\lambda^{n+1}}
=\frac{1}{\lambda}\sum_{n=0}^\infty \int_0^t \frac{(t-s)^{n}}{\lambda^n\cdot n!}f(s)\rd s=
\frac{1}{\lambda}\int_0^te^{\frac{t-s}{\lambda}}f(s)\rd s,\]
where $t\in[0,1], f\in L^2[0,1].$
\end{proof}

About the lower bound of $\|(\lambda-V)^{-1}Vf\|$, we obtain the following estimate.

\begin{lem}\label{Vf} For $u\in(0,1]$,
\[\|(\lambda-V)^{-1}Vf\|\geq\frac{1}{\sqrt{u}}(|\int_0^u f(s)e^{\frac{u-s}{\lambda}}\rd s|-\|f\|).\]
\end{lem}

\begin{proof} By the Cauchy-Schwarz inequality,
\begin{align*}\|(\lambda-V)^{-1}Vf\|^2&=\int_0^1|(\lambda-V)^{-1}Vf(t)|^2\rd t\\
&\geq\int_0^u|(\lambda-V)^{-1}Vf(t)|^2\rd t\\
&\geq\frac{1}{u}|\int_0^u(\lambda-V)^{-1}Vf(t)\rd t|^2.
\end{align*}
Then by Fubini's theorem,
\begin{align*}\int_0^u(\lambda-V)^{-1}Vf(t)\rd t&=\int_0^u\frac{1}{\lambda}\int_0^tf(s)e^{\frac{t-s}{\lambda}}\rd s\rd t\\
&=\int_0^u\frac{f(s)}{\lambda}\int_s^ue^{\frac{t-s}{\lambda}}\rd t\rd s\\
&=\int_0^u f(s)(e^{\frac{u-s}{\lambda}}-1)\rd s\\
&=\int_0^u f(s)e^{\frac{u-s}{\lambda}}\rd s-Vf(u).
\end{align*}
So according to the fact that $\|V\|\leq 1$ (precisely $\|V\|=\frac{2}{\pi}$ \cite[Theorem 5.8]{Dav88}), we have
\[\|(\lambda-V)^{-1}Vf\|\geq\frac{1}{\sqrt{u}}(|\int_0^u f(s)e^{\frac{u-s}{\lambda}}\rd s|-\|f\|).\]
\end{proof}

\begin{lem}\label{fe}
Let $0\neq f\in C[0,1]$. Then there exist a sequence $\{\lambda_k\}_{k=1}^\infty$ of positive numbers with $\lim\limits_{k\rightarrow \infty}\lambda_k=0$ and a positive number $d$, such that
\[\|(\lambda_k-V)^{-1}f\|\geq e^{\frac{d}{\lambda_k}},~~\forall k\geq 1.\]
\end{lem}

\begin{proof}
Pick $u\in(0,1]$ so that $f(u)\ne 0$. By Lemma \ref{c-Phi-e-k}, we can choose a positive number $d_1$ and a sequence $\{\rho_k\}_{k=1}^\infty$ of positive numbers with $\lim\limits_{k\rightarrow \infty}\rho_k=+\infty$ so that $|\Phi_{f, u}(\rho_k)|>e^{d_1\rho_k}$. Set $d=\frac{d_1}{2}$. Choose $N\in\bN$ so that
\[e^{d\rho_k}>(1+\|f\|),~\forall k\geq N.\]
Set $\lambda_k=\frac{1}{\rho_{k+N}}$. Then from the fact that $\|V\|\leq 1$ and Lemma \ref{Vf},
it follows that
\begin{align*}\|(\lambda_k-V)^{-1}f\|
&\geq \|V\|\|(\lambda_k-V)^{-1}f\|\\
&\geq
\|(\lambda_k-V)^{-1}Vf\|\\
&\geq \frac{1}{\sqrt{u}}(|\int_0^u f(s)e^{\frac{u-s}{\lambda_k}}\rd s|-\|f\|)\\
&\geq\frac{1}{\sqrt{u}}(|\int_0^u f(s)e^{(u-s)\rho_{k+N}}\rd s|-\|f\|)\\
&=\frac{1}{\sqrt{u}}(|\Phi_{f,u}(\rho_{k+N})|-\|f\|)\\
&\geq \frac{1}{\sqrt{u}}(e^{2d\rho_{k+N}}-\|f\|)\\
&= \frac{1}{\sqrt{u}}(e^{d\rho_{(k+N)}}e^{d\rho_{(k+N)}}-\|f\|)\\
&\geq \frac{1}{\sqrt{u}}[e^{\frac{d}{\lambda_k}}(1+\|f\|)-\|f\|]\\
&>e^\frac{d}{\lambda_k}.
\end{align*}
\end{proof}

The following fact is known. For completeness, we include a proof here.

\begin{lem}\label{norm of resolvent}
For $\lambda\neq 0$,
\[\|(\lambda-V)^{-1}\|\leq \frac{1}{|\lambda|}e^{\frac{1}{|\lambda|}}.\]
\end{lem}

\begin{proof}
By \cite[Corollary 5.2]{Ker99},
\[\|V^n\|\leq \frac{1}{n!},~~n\in \bN.\]
Therefore, for $\lambda\neq 0$,
\[\|(\lambda-V)^{-1}\|\leq \frac{1}{|\lambda|}\sum_{n=0}^\infty \frac{\|V\|^n}{|\lambda|^n}\leq \frac{1}{|\lambda|}e^{\frac{1}{|\lambda|}}.\]
\end{proof}

Now we have all the pieces in place to prove the power set of $V$ does not contain $0$!

\begin{thm}\label{0 is not in the power set of V}
Let $V$ be the Volterra operator on $L^2[0,1]$. Then for any nonzero vector $f$ in $L^2[0,1]$, $k_f>0$.
In other words, $0\notin \Lambda(V)$.
\end{thm}

\begin{proof}

By Lemma \ref{fe}, if $f\neq 0$ is continuous, then there exist positive numbers $d$ and $\lambda_k$ with $\lim\limits_{k\rightarrow \infty}\lambda_k=0$ such that
\[\ln\|(\lambda_k-V)^{-1}f\|\geq \frac{d}{\lambda_k}.\]
On the other hand, by Lemma \ref{norm of resolvent},
\[\|(\lambda_k-V)^{-1}\|\leq \frac{1}{|\lambda_k|}e^{\frac{1}{|\lambda_k|}}.\]
Thus,  \[\frac{\ln\|(\lambda_k-V)^{-1}f\|}{\ln\|(\lambda_k-V)^{-1}\|}\geq\frac{\frac{d}{|\lambda_k|}}{\frac{1}{|\lambda_k|}+\ln\frac{1}{|\lambda_k|}}.\]
 So,
\[k_f\geq d>0.\]

 For any nonzero $f\in L^2[0,1]$, as $Vf$ is nonzero and continuous, by Lemma \ref{commute},
\[k_f\geq k_{Vf}>0.\]
\end{proof}

\begin{prop}\label{key estimate}
For $\alpha\in [0,1)$, let $f_\alpha$ be the characteristic function of $[\alpha,1]$. Then for $\lambda<0$,
\[\|(\lambda-V)^{-1}f_\alpha\|^2=\frac{1}{-2\lambda}[1-e^{\frac{2}{\lambda}(1-\alpha)}].\]
In particular,
\[\underset{\lambda\to 0^-}{\lim}\frac{\ln\|(\lambda-V)^{-1}f_\alpha\|}{\ln \|(\lambda-V)^{-1}f_0\|}= 1.\]
\end{prop}

\begin{proof}
Recall that
\[V^nf_\alpha(t)=\int_0^t \frac{(t-s)^{n-1}}{(n-1)!}f_\alpha(s)\rd s,~~f\in L^2[0,1],~n\geq 1.\]
Thus, for $n\geq 1$,
\[V^nf_\alpha(t)=\begin{cases}0,&\quad\textrm{ if } t\leq \alpha;\\
\\
\frac{(t-\alpha)^n}{n!},&\quad\textrm{ if } \alpha\leq t\leq 1.\end{cases}\\\]
Therefore,
\begin{enumerate}
\item For $0\leq t<\alpha$, we have
\[(\lambda-V)^{-1}f_\alpha(t)=\frac{1}{\lambda}[f_\alpha(t)+\sum\limits_{n=1}^\infty \frac{V^nf_\alpha(t)}{\lambda^n}]=0.\]
\item For $\alpha\leq t\leq 1$, it follows that
\[(\lambda-V)^{-1}f_\alpha(t)=\frac{1}{\lambda}[f_\alpha(t)+\sum\limits_{n=1}^\infty \frac{V^nf_\alpha(t)}{\lambda^n}]=\frac{1}{\lambda}[1+\sum\limits_{n=1}^\infty \frac{(t-\alpha)^n}{\lambda^n n!}]=\frac{1}{\lambda}e^{\frac{t-\alpha}{\lambda}}.\]
\end{enumerate}
For $\lambda<0$,
\[\|(\lambda-V)^{-1}f_\alpha\|^2=\frac{1}{|\lambda|^2}\int_{\alpha}^1 e^{\frac{2}{\lambda}(t-\alpha)}\rd t= \frac{\lambda}{2|\lambda|^2}[e^{\frac{2}{\lambda}(1-\alpha)}-1]=\frac{1}{-2\lambda}[1-e^{\frac{2}{\lambda}(1-\alpha)}].\]
Thus,
\[\underset{\lambda\to 0^-}{\lim}\frac{\ln\|(\lambda-V)^{-1}f_\alpha\|}{\ln \|(\lambda-V)^{-1}f_0\|}=\underset{\lambda\to 0^-}{\lim}\frac{\ln{\frac{1}{-2\lambda}[1-e^{\frac{2}{\lambda}(1-\alpha)}]}}{\ln{\frac{1}{-2\lambda}[1-e^{\frac{2}{\lambda}(1-0)}]}}=1.\]

\end{proof}

%原来的证明
%\begin{proof}
%Recall that
%\[V^nf(t)=\int_0^t \frac{(t-s)^{n-1}}{(n-1)!}f(s)\rd s,~~f\in L^2[0,1],~n\geq 1.\]
%Thus,
%\[(\lambda-V)^{-1}f_\alpha(t)=\sum_{n=0}^\infty\frac{V^nf_\alpha(t)}{\lambda^{n+1}}=\frac{f_\alpha(t)}{\lambda}+\sum_{n=1}^\infty\frac{1}{\lambda^{n+1}}\int_0^t  \frac{(t-s)^{n-1}}{(n-1)!}f_\alpha(s)\rd s=\frac{f_\alpha(t)}{\lambda}+\frac{1}{\lambda^2}\int_{0}^te^{\frac{t-s}{\lambda}}f_\alpha(s)\rd s.\]
%Therefore,
%\begin{enumerate}
%\item For $0\leq t<\alpha$, $(\lambda-V)^{-1}f_\alpha(t)=0$.
%\item For $\alpha\leq t\leq 1$, it follows that \[(\lambda-V)^{-1}f_\alpha(t)=\frac{1}{\lambda}+\frac{1}{\lambda^2}\int_{\alpha}^te^{\frac{t-s}{\lambda}}\rd s=\frac{1}{\lambda}e^{\frac{t-\alpha}{\lambda}}.\]
%\end{enumerate}
%For $\lambda<0$,
%\[\|(\lambda-V)^{-1}f_\alpha\|^2=\frac{1}{|\lambda|^2}\int_{\alpha}^1 e^{\frac{2}{\lambda}(t-\alpha)}\rd t= \frac{\lambda}{2|\lambda|^2}[e^{\frac{2}{\lambda}(1-\alpha)}-1]=\frac{1}{-2\lambda}[1-e^{\frac{2}{\lambda}(1-\alpha)}].\]
%Thus,
%\[\underset{\lambda\to 0^-}{\lim\sup}\frac{\ln\|(\lambda-V)^{-1}f_\alpha\|}{\ln \|(\lambda-V)^{-1}f_0\|}=\underset{\lambda\to 0^-}{\lim\sup}\frac{\ln{\frac{1}{-2\lambda}[1-e^{\frac{2}{\lambda}(1-\alpha)}]}}{\ln{\frac{1}{-2\lambda}[1-e^{\frac{2}{\lambda}(1-0)}]}}=1.\]
%
%\end{proof}

We now proceed to make some comments. For $\alpha\in (0,1)$, as shown in Proposition \ref{key estimate},
\[\underset{\lambda\to 0^-}{\lim}\frac{\ln\|(\lambda-V)^{-1}f_\alpha\|}{\ln \|(\lambda-V)^{-1}f_0\|}= 1.\]
It seems that Douglas and Yang didn't complete the assertion that $(0,1)\subset\Lambda(V)$ if one examining the proof of \cite[Proposition 6.2]{DY3}.
Some extra work will be needed.
%In fact, by \cite[Theorem 2.1]{Ran22},
%\[\|(\lambda-V)^{-1}\|=\frac{1}{-\lambda},~~\forall \lambda<0,\]
%so by Proposition \ref{key estimate},
%\[k_{f_\alpha}\geq
%%\underset{\lambda\to 0^-}{\lim\sup}\frac{\ln\|(\lambda-V)^{-1}f_\alpha\|}{\ln \|(\lambda-V)^{-1}\|}=
%\underset{\lambda\to 0^-}{\lim\sup}\frac{\ln\|(\lambda-V)^{-1}f_\alpha\|}{\ln \|(\lambda-V)^{-1}\|}
%=\underset{\lambda\to 0^-}{\lim\sup}\frac{\frac{1}{2}\ln \frac{1}{-2\lambda}[1-e^{\frac{2}{\lambda}(1-\alpha)}]}{\ln \frac{1}{-\lambda}}=\frac{1}{2}.\]Hence,
%$\{k_{f_\alpha},\alpha\in [0,1)\}\subset [\frac{1}{2},1]$.
%However, what is surprising is the following.

%However, our method used in proving Proposition \ref{power set of V} is influenced by theirs.
%Some extra work is needed to figure out that $\Lambda(V)=(0,1]$ although $0\notin \Lambda(V)$ by Theorem \ref{0 is not in the power set of V}.

%\begin{lem}\label{h}
%For $0\leq \alpha<1$, let $f_\alpha$ be the characteristic function of $[\alpha,1]$. Then\\
%\hspace*{40pt}$(\lambda-V)^{-1}Vf_\alpha(t)=\begin{cases}0,&\quad\textrm{ if } t\leq \alpha;\\
%e^{\frac{t-\alpha}{\lambda}}-1,&\quad\textrm{ if } \alpha\leq t\leq 1.\end{cases}$\\
%\end{lem}
%
%\begin{proof}
%The results could be obtained easily if one notice that
%\[(\lambda-V)^{-1}Vf_\alpha(t)=
%\frac{1}{\lambda}\int_0^te^{\frac{t-s}{\lambda}}f_\alpha(s)\rd s.\]
%\end{proof}

\begin{prop}\label{power set of V}
For $\alpha\in [0,1)$, let $f_\alpha$ be the characteristic function of $[\alpha,1]$ and $g_\alpha\doteq Vf_\alpha$. Then $$\{k_{g_\alpha}:\alpha\in [0,1)\}=(0,1].$$
\end{prop}

\begin{proof}

From Lemma \ref{metric at Vf},
\[(\lambda-V)^{-1}Vf_\alpha(t)=
\frac{1}{\lambda}\int_0^te^{\frac{t-s}{\lambda}}f_\alpha(s)\rd s,~~t\in [0,1].\]
Then it is straightforward to check that
\[(\lambda-V)^{-1}Vf_\alpha(t)=\begin{cases}0,&\quad\textrm{ if } t\leq \alpha;\\
\\
e^{\frac{t-\alpha}{\lambda}}-1,&\quad\textrm{ if } \alpha\leq t\leq 1.\end{cases}\\\]
For $\lambda\neq 0$, set $h_{\alpha,\lambda}\doteq (\lambda-V)^{-1}g_\alpha$.
Then ,
\[\|h_{\alpha,\lambda}\|^2=\int_{\alpha}^1 |e^{\frac{t-\alpha}{\lambda}}-1|^2 \rd t.\]
For $\lambda\in \bC$, let $\overline{\lambda}$, $\textup{Re}(\lambda)$ and $\textup{Im}(\lambda)$
be the complex conjugate, the real part and imaginary part of $\lambda$ respectively.

We first state some observations.
\begin{enumerate}
\item For $\lambda>0$,
\[\|h_{\alpha,\lambda}\|^2=\int_{\alpha}^1 [e^{\frac{2}{\lambda}(t-\alpha)}-2e^{\frac{(t-\alpha)}{\lambda}}+1]\rd t
=\frac{\lambda}{2}e^{\frac{2}{\lambda}(1-\alpha)}-2\lambda e^{\frac{1}{\lambda}(1-\alpha)}+[\frac{3\lambda}{2}+(1-\alpha)].\]
Then there exists $\delta(\alpha)>0$, such that
\[\|h_{\alpha,\lambda}\|^2\geq \frac{\lambda}{4}e^{\frac{2}{\lambda}(1-\alpha)},~~\textup{if}~0<\lambda<\delta(\alpha).\]
\item For $\lambda\neq 0$ and $\textup{Re}(\lambda)\leq 0$, \[|e^{\frac{t-\alpha}{\lambda}}|=e^{\frac{\textup{Re}(\lambda)(t-\alpha)}{|\lambda|^2}}\leq 1,~~ \textup{if}~~t\in [\alpha,1].\]
    Therefore, $\|h_{\alpha,\lambda}\|^2\leq 4$.
\item  For $\lambda\neq 0$,
\[\|h_{\alpha,\lambda}\|^2=\int_{\alpha}^1 [e^{2\textup{Re}(\frac{1}{\lambda})(t-\alpha)}-e^{\frac{(t-\alpha)}{\lambda}}-e^{\frac{(t-\alpha)}{\overline{\lambda}}}+1]\rd t\]
\[=\frac{1}{2\textup{Re}(\frac{1}{\lambda})}[e^{2\textup{Re}(\frac{1}{\lambda})(1-\alpha)}-1]-\lambda[e^{\textup{Re}(\frac{1}{\lambda})(1-\alpha)}e^{i\textup{Im}(\frac{1}{\lambda})(1-\alpha)}-1]
-\overline{\lambda}[e^{\textup{Re}(\frac{1}{\lambda})(1-\alpha)}e^{-i\textup{Im}(\frac{1}{\lambda})(1-\alpha)}-1]+(1-\alpha)~~~~~~(*).\]
\end{enumerate}

Fix $\alpha\in [0,1)$. We will show that $k_{g_\alpha}=1-\alpha$.
\begin{itemize}
\item By the above-mentioned (1) and Lemma \ref{norm of resolvent},
\[k_{g_\alpha}\geq \underset{\lambda\to 0^+}{\lim\sup}\frac{\ln \|h_{\alpha,\lambda}\|}{\ln \|(\lambda-V)^{-1}\|}
\geq \underset{\lambda\to 0^+}{\lim\sup}\frac{\frac{1}{2}\ln [\frac{\lambda}{4}e^{\frac{2}{\lambda}(1-\alpha)}]}{\ln [\frac{1}{\lambda}e^{\frac{1}{\lambda}}]}=1-\alpha.\]
\item Set $\Omega\doteq  \{\lambda:\lambda\in \bC,\textup{Re}(\lambda)>0\}$. Then according to the above-mentioned (2) and the fact that
\[\|(\lambda-V)^{-1}\|\geq \frac{1}{|\lambda|},~~\forall \lambda\neq 0,\]
it follows that
\[k_{g_\alpha}=\underset{\Omega\ni\lambda\to 0}{\lim\sup}\frac{\ln \|h_{\alpha,\lambda}\|}{\ln \|(\lambda-V)^{-1}\|}.\]
Choose a sequence $\{\lambda_n:n\geq 1\}\subset \Omega$
such that $\underset{n\to \infty}{\lim}\lambda_n=0$ and
\[k_{g_\alpha}=\underset{n\to \infty}{\lim}\frac{\ln \|h_{\alpha,\lambda_n}\|}{\ln \|(\lambda_n-V)^{-1}\|}.\]
Let $M\doteq \underset{n\in \bN}{\sup} \textup{Re}(\frac{1}{\lambda_n})$.
 Pick a subsequence $\{\lambda_{n_k}:k\geq 1\}$ of $\{\lambda_n: n\geq 1\}$
such that $\underset{k\to \infty}{\lim}\textup{Re}(\frac{1}{\lambda_{n_k}})=M$.
\item We claim that $M=+\infty$.
Otherwise, suppose that $M<+\infty$.
Note that from the fact that $\underset{k\to \infty}{\lim}\lambda_{n_k}=0$,
\begin{align*}\underset{k\to \infty}{\lim\sup}|\lambda_{n_k} [e^{\textup{Re}(\frac{1}{\lambda_{n_k}})(1-\alpha)}e^{i\textup{Im}(\frac{1}{\lambda_{n_k}})(1-\alpha)}-1]|
&\leq
\underset{k\to \infty}{\lim\sup}|\lambda_{n_k}|[e^{\textup{Re}(\frac{1}{\lambda_{n_k}})(1-\alpha)}+1]\\
&\leq \underset{k\to \infty}{\lim\sup}|\lambda_{n_k}|\underset{k\to \infty}{\lim\sup}[e^{\textup{Re}(\frac{1}{\lambda_{n_k}})(1-\alpha)}+1]\\
&\leq0\cdot [e^{M(1-\alpha)}+1]\\
&=0;
\end{align*}
and similarly,
\[\underset{k\to \infty}{\lim\sup}|\overline{\lambda_{n_k}} [e^{\textup{Re}(\frac{1}{\lambda_{n_k}})(1-\alpha)}e^{-i\textup{Im}(\frac{1}{\lambda_{n_k}})(1-\alpha)}-1]|
=0.\]
Then by using $(*)$, we have
\[\underset{k\to \infty}{\lim}\|h_{\alpha,\lambda_{n_k}}\|^2=\frac{1}{2M}[e^{2M(1-\alpha)}-1]+(1-\alpha)\in (0, +\infty).\]
So by
\[\|(\lambda_{n_k}-V)^{-1}\|\geq \frac{1}{|\lambda_{n_k}|},~~k\in \bN,\]
we have
\[k_{g_\alpha}=\underset{k\to \infty}{\lim}\frac{\ln \|h_{\alpha,\lambda_{n_k}}\|}{\ln \|(\lambda_{n_k}-V)^{-1}\|}=0<1-\alpha.\]
This contradiction implies that $M\doteq \underset{n\in \bN}{\sup}\textup{Re}(\frac{1}{\lambda_n})=+\infty$ holds. In other words,
$\underset{k\to \infty}{\lim}\textup{Re}(\frac{1}{\lambda_{n_k}})=+\infty$.
\item By $(*)$, there exists  $N\in \bN$, such that for $k\geq N$,
\[\|h_{\alpha,\lambda_{n_k}}\|^2\leq e^{2\textup{Re}(\frac{1}{\lambda_{n_k}})(1-\alpha)}+2e^{\textup{Re}(\frac{1}{\lambda_{n_k}})(1-\alpha)}+3\leq 4e^{2\textup{Re}(\frac{1}{\lambda_{n_k}})(1-\alpha)},\]
\[\|h_{\alpha,\lambda_{n_k}}\|^2\geq \frac{1}{2\textup{Re}(\frac{1}{\lambda_{n_k}})}e^{2\textup{Re}(\frac{1}{\lambda_{n_k}})(1-\alpha)}
-2e^{\textup{Re}(\frac{1}{\lambda_{n_k}})(1-\alpha)}-3\geq \frac{1}{4\textup{Re}(\frac{1}{\lambda_{n_k}})}e^{2\textup{Re}(\frac{1}{\lambda_{n_k}})(1-\alpha)}\gg 10;\]
and
\[\|h_{0,\lambda_{n_k}}\|^2\geq \frac{1}{2\textup{Re}(\frac{1}{\lambda_{n_k}})}e^{2\textup{Re}(\frac{1}{\lambda_{n_k}})}
-2e^{\textup{Re}(\frac{1}{\lambda_{n_k}})}-3\geq \frac{1}{4\textup{Re}(\frac{1}{\lambda_{n_k}})}e^{2\textup{Re}(\frac{1}{\lambda_{n_k}})}\gg 10.\]
Therefore, after observing that
\[\|(\lambda_{n_k}-V)^{-1}\|\geq \|(\lambda_{n_k}-V)^{-1}Vf_0\|=\|h_{0, \lambda_{n_k}}\|,\]
we have
\begin{align*}
k_{g_\alpha}& =\underset{k\to \infty}{\lim}\frac{\ln \|h_{\alpha,\lambda_{n_k}}\|}{\ln \|(\lambda_{n_k}-V)^{-1}\|}\\
	&\leq \underset{k\to \infty}{\lim\sup}\frac{\ln \|h_{\alpha,\lambda_{n_k}}\|}{\ln \|h_{0,\lambda_{n_k}}\|}\\
&\leq \underset{k\to \infty}{\lim\sup}\frac{\ln [4e^{2\textup{Re}(\frac{1}{\lambda_{n_k}})(1-\alpha)}]}{\ln [\frac{1}{4\textup{Re}(\frac{1}{\lambda_{n_k}})}e^{2\textup{Re}(\frac{1}{\lambda_{n_k}})}]}\\
&=\underset{k\to \infty}{\lim\sup} \frac{2\textup{Re}(\frac{1}{\lambda_{n_k}})(1-\alpha)+\ln 4}{2\textup{Re}(\frac{1}{\lambda_{n_k}})-\ln[4\textup{Re}(\frac{1}{\lambda_{n_k}})]}\\
&= \underset{k\to \infty}{\lim\sup} \frac{(1-\alpha)+\frac{\ln 4}{2\textup{Re}(\frac{1}{\lambda_{n_k}})}}{1-\frac{\ln[4\textup{Re}(\frac{1}{\lambda_{n_k}})]}{2\textup{Re}(\frac{1}{\lambda_{n_k}})}}\\
&=1-\alpha.
\end{align*}

%\[k_{g_\alpha}=\underset{k\to \infty}{\lim}\frac{\ln \|h_{\alpha,\lambda_{n_k}}\|}{\ln \|(\lambda_{n_k}-V)^{-1}\|}\leq
%\underset{k\to \infty}{\lim\sup}\frac{\ln \|h_{\alpha,\lambda_{n_k}}\|}{\ln \|h_{0,\lambda_{n_k}}\|}
%\leq\underset{k\to \infty}{\lim\sup} \frac{2\textup{Re}(\frac{1}{\lambda_{n_k}})(1-\alpha)+\ln 4}{2\textup{Re}(\frac{1}{\lambda_{n_k}})-\ln[4\textup{Re}(\frac{1}{\lambda_{n_k}})]} =1-\alpha.\]
\end{itemize}
In summary, $k_{g_\alpha}=1-\alpha$.

Therefore,
$$\{k_{g_\alpha}:\alpha\in [0,1)\}=(0,1].$$

\end{proof}

Finally, we are at the position to obtain Theorem \ref{C}.

\begin{proof}[\textup{\textbf{Proof of Theorem \ref{C}}}]
Since $\Lambda(V)\subset [0,1]$, the conclusion now follows from Theorem \ref{0 is not in the power set of V} and
Proposition \ref{power set of V}.
\end{proof}

%%%%%%%%%%%%%%%  REFERENCES  %%%%%%%%%%%%%%%%

%\bibliographystyle{amsplain}
%\bibliography{reference}

%
%\newpage

\end{document}